\documentclass[preprint,12pt]{elsarticle}
\usepackage{color}

\usepackage{amssymb}
\usepackage{amsthm}
\usepackage{amsmath}
\usepackage{hyperref}
\usepackage{bbm}

\newtheorem{theorem}{Theorem}[section]
\newtheorem{lemma}[theorem]{Lemma}

\newtheorem{problem}[theorem]{Open problem}
\theoremstyle{remark}

\numberwithin{equation}{section}
\allowdisplaybreaks

\begin{document}

\begin{frontmatter}



\title{%
Bounded compact and dual compact\\ 
approximation properties of Hardy spaces:\\ 
new results and open problems}


\author[a]{Oleksiy Karlovych}
\ead{oyk@fct.unl.pt}
\ead[url]{https://docentes.fct.unl.pt/oyk}

\address[a]{Centro de Matem{\'a}tica e Aplica{\c{c}\~o}es,  
Departamento de Matematica,\\
Faculdade de Ciencias e Tecnologia,
Universidade Nova de Lisboa,\\
Quinta da Torre, Caparica,
2829--516, Portugal}

\author[b]{Eugene Shargorodsky}
\ead{eugene.shargorodsky@kcl.ac.uk}
\ead[url]{https://www.kcl.ac.uk/people/eugene-shargorodsky}

\address[b]{
Department of Mathematics, 
King's College London, \\
Strand, London, 
WC2R 2LS,
United Kingdom
}

%%%----------------------------------------------------------------------------
\begin{abstract}
The aim of the paper is to highlight some open problems concerning 
approximation properties of Hardy spaces. We also present some results on 
the bounded compact and the dual compact approximation properties (shortly, 
BCAP and DCAP) of such spaces, to provide background for the open problems.
Namely, we consider abstract Hardy spaces $H[X(w)]$ built 
upon translation-invariant Banach function spaces $X$ with weights $w$
such that $w\in X$ and $w^{-1}\in X'$, where $X'$ is the associate space 
of $X$. We prove that if $X$ is separable, then $H[X(w)]$ has the BCAP
with the approximation constant $M(H[X(w)])\le 2$. Moreover, if $X$ is 
reflexive, then $H[X(w)]$ has the BCAP and the DCAP with the approximation
constants $M(H[X(w)])\le 2$ and $M^*(H[X(w)])\le 2$, respectively.
In the case of classical weighted Hardy space $H^p(w) = H[L^p(w)]$ 
with $1<p<\infty$, one has a sharper result: $M(H^p(w))\le 2^{|1-2/p|}$ and 
$M^*(H^p(w))\le 2^{|1-2/p|}$.
\end{abstract}




\end{frontmatter}

%\linenumbers
%%%----------------------------------------------------------------------------
\section{Introduction}
For a Banach space $E$, let $\mathcal{B}(E)$ and $\mathcal{K}(E)$ 
denote the sets of bounded linear and compact linear operators on $E$, 
respectively. The norm of an operator $A\in\mathcal{B}(E)$ is denoted
by $\|A\|_{\mathcal{B}(E)}$. The essential norm of $A \in \mathcal{B}(E)$ 
is defined as follows:
\[
\|A\|_{\mathcal{B}(E),\mathrm{e}} 
:= 
\inf\{\|A - K\|_{\mathcal{B}(E)}\ : \  K \in \mathcal{K}(E)\}.
\]
For a Banach space $E$ and an operator $A\in\mathcal{B}(E)$, consider
the following measure of noncompactness:
\[
\|A\|_{\mathcal{B}(E),m} 
:= 
\inf_{{\tiny \begin{array}{c}
L \subseteq E \mbox{ closed linear subspace}   \\
\mathrm{dim}  (E/L) < \infty  
\end{array} }} \big\|A|_L\big\|_{\mathcal{B}(L)} ,
\]
where $A|_L$ denotes the restriction of $A$ to $L$. 

It follows from \cite[formula (3.29)]{LS71} that if $A\in\mathcal{B}(E)$, then
%%%
\begin{equation}\label{eq:Lebow-Schechter}
\|A\|_{\mathcal{B}(E),m}
\le 
\|A\|_{\mathcal{B}(E),\mathrm{e}}.
\end{equation}
%%%
Motivated by applications to the Fredholm theory of Toeplitz operators
(see \cite{S21}), we are interested in the smallest constant $C$ in the 
reverse estimate:
%%%
\begin{equation}\label{eq:dream}
\|A\|_{\mathcal{B}(E),\mathrm{e}}
\le 
C\|A\|_{\mathcal{B}(E),m}
\quad\mbox{for all}\quad A\in\mathcal{B}(E).
\end{equation}
Note that such estimate is not true without additional assumptions on $E$
(see \cite{AT87} and also \cite{KS22}).

A Banach space $E$ is said to have the bounded compact approximation property 
(BCAP) if there exists a constant $M \in (0, \infty)$ such that given any 
$\varepsilon > 0$ and any finite set $F \subset E$, there exists an operator
$T \in \mathcal{K}(E)$ such that
%%% 
\begin{equation}\label{eq:BCAP}
\|I - T\|_{\mathcal{B}(E)} \le M, 
\quad  
\|y - Ty\|_E < \varepsilon  
\quad \mbox{for all}\quad 
y \in F .
\end{equation}
%%%
Here $I$ is the identity map from $E$ to itself. The greatest lower bound 
of the constants $M$ for which \eqref{eq:BCAP} holds will be denoted by 
$M(E)$. 

A Banach space $E$ with the dual space $E^*$ is said to have the dual compact
approximation property (DCAP) if there is a constant $M^*\in(0,\infty)$ 
such that given any $\varepsilon>0$ and any finite set $G\subset E^*$
there exists an operator $T\in\mathcal{K}(E)$ such that
%%%
\begin{equation}\label{eq:DCAP}
\|I-T\|_{\mathcal{B}(E)}\le M^*,
\quad
\|z-T^*z\|_{E^*}<\varepsilon
\quad\mbox{for all}\quad
z\in G.
\end{equation}
%%%
The greatest lower bound of the constants $M^*$, for which \eqref{eq:DCAP}
holds, will be denoted by $M^*(E)$.

It is easy to see that if $E$ is reflexive, then $E$ has the DCAP if and
only if its dual space $E^*$ has the BCAP. In this case $M^*(E)=M(E^*)$. 
%%%----------------------------------------------------------------------------
\begin{theorem}\label{th:BCAP-DCAP} 
Let $E$ be a Banach space. 
\begin{enumerate}
\item[{\rm(a)}]
If $E$ has the BCAP, then \eqref{eq:dream} holds with $C=2M(E)$.

\item[{\rm(b)}]
If $E$ has the DCAP, then \eqref{eq:dream} holds with $C=M^*(E)$.
\end{enumerate}
\end{theorem}
%%%----------------------------------------------------------------------------
Part (a) follows from \cite[Theorems~3.1 and 3.6]{LS71}
(note that there is a typo in \cite[formula (3.7)]{LS71}, where the factor $2$
is missing). Part (b) was proved in \cite[Theorem 2.2]{S21}.

It follows from \eqref{eq:Lebow-Schechter} and Theorem~\ref{th:BCAP-DCAP}
that if a Banach space $E$ has the BCAP or the DCAP, then the essential norm 
$\|\cdot\|_{\mathcal{B}(E),\mathrm{e}}$ and the $m$-measure of noncompactness
$\|\cdot\|_{\mathcal{B}(E),m}$ are equivalent.

The condition $\|I - T\|_{\mathcal{B}(E)} \le M$ is often 
substituted by $\|T\|_{\mathcal{B}(E)} \le M$ in the
definition of BCAP (see, e.g., \cite{C01,CK90,LT77,LT79}, 
and the references therein). Let $m(E)$ be the greatest lower 
bound of the constants $M$ for which the conditions in this alternative 
definition of BCAP are satisfied. Clearly,
\[
m(E) - 1 \le M(E) \le m(E) + 1 .
\]
We are interested in $M(E)$ rather than in $m(E)$ 
because the former appears naturally in estimates for the essential norms of 
operators by their measures of noncompactnes 
(see Theorem~\ref{th:BCAP-DCAP} and \cite{AT87,ES05,LS71,S21}). 
It is well known that $m(L^p[0, 1]) = 1$, $1 \le p < \infty$ (see, e.g.,
\cite[Lemma~19.3.5]{P80}). The value of $M(L^p[0,1])$ was found in 
\cite[Theorem~3.2]{SS23}: if $1\le p<\infty$, then $M(L^p[0,1])=C_p$, where $C_p$ 
is the norm of the operator
\[
L^p[0, 1] \ni f \ \longmapsto \ f - \int_0^1 f(t)\, dt \in L^p[0, 1] ,
\]
i.e. $C_1=2$ and, for $1<p<\infty$, 
%%%
\begin{equation}\label{eq:Franchetti}
C_p:=\max_{0\le\alpha\le 1}
\left(\alpha^{p-1}+(1-\alpha)^{p-1}\right)^{1/p}
\left(\alpha^{1/(p-1)}+(1-\alpha)^{1/(p-1)}\right)^{1-1/p}
\end{equation}
%%%
(see \cite[formula (8)]{F90}). 

For a function $f \in L^1$ on the unit circle 
$\mathbb{T} :=\left\{z \in \mathbb{C} : \ |z| = 1\right\}$, let
\[
\widehat{f}(n) 
= 
\frac{1}{2\pi} \int_{-\pi}^\pi 
f\left(e^{i\theta}\right) e^{-i n\theta}\, d\theta , 
\quad 
n \in \mathbb{Z}
\]
be the Fourier coefficients of $f$. Let $X$ be a Banach space of 
measurable complex-valued functions on $\mathbb{T}$ continuously embedded 
into $L^1$. Let
\[
H[X] := 
\left\{
g\in X\ :\ \widehat{g}(n)=0\quad\mbox{for all}\quad n<0
\right\} 
\]
denote the abstract Hardy space built upon the space $X$. 
In the case $X = L^p$, where $1\le p\le\infty$, we will use the standard 
notation $H^p := H[L^p]$.

The classical Hardy spaces $H^p$ with $1<p<\infty$ have the 
BCAP and the DCAP with
%%%
\begin{equation}\label{eq:Shargorodsky-estimates}
M(H^p)\le 2^{|1-2/p|},
\quad
M^*(H^p)\le 2^{|1-2/p|}
\end{equation}
%%%
(see  \cite[Theorem~3.1]{S21}).

A measurable function $w:\mathbb{T}\to[0,\infty]$ is said to be a weight
if $0<w<\infty$ a.e. on $\mathbb{T}$. Let $1<p<\infty$ and $w$ be a weight.
Weighted Lebesgue spaces $L^p(w)$ consist of all measurable functions
$f:\mathbb{T}\to\mathbb{C}$ such that $fw\in L^p$. The norm in $L^p(w)$
is defined by
\[
\|f\|_{L^p(w)}:=\|fw\|_{L^p}
=
\left(\int_{\mathbb{T}}|f(t)|^pw^p(t)\,dm(t)\right)^{1/p},
\]
where $m$ is the Lebesgue measure on $\mathbb{T}$ normalized so that
$m(\mathbb{T})=1$. Estimates \eqref{eq:Shargorodsky-estimates} remain true 
for the weighted Hardy spaces $H^p(w):=H[L^p(w)]$.
%%%-----------------------------------------------------------------------------
\begin{theorem}
\label{th:BCAP-DCAP-weighted-Hardy}
Let $1<p<\infty$, $1/p+1/p'=1$, and let $w$ be a weight such that $w\in L^p$
and $1/w\in L^{p'}$. Then the weighted Hardy space $H^p(w)$ has 
the BCAP and the DCAP with
\[
M(H^p(w))\le2^{|1-2/p|},
\quad
M^*(H^p(w))\le 2^{|1-2/p|}.
\]
\end{theorem}
%%%----------------------------------------------------------------------------
Let $X$ be a Banach function space on the unit circle $\mathbb{T}$
equipped with the Lebesgue measure $dm$ and let $X'$ be its associate space
(see \cite[Ch.~1]{BS88}). We postpone the definitions of these notions until 
Section~\ref{sec:BFS}. Here we only mention that the class of Banach function
spaces is very rich, it includes all Lebesgue spaces $L^p$, $1\le p\le\infty$,
Orlicz spaces $L^\varphi$ (see, e.g., \cite[Ch.~4, Section~8]{BS88}), and
Lorentz spaces $L^{p,q}$ (see, e.g., \cite[Ch.~4, Section~4]{BS88}). 
For a weight $w$, the weighted space $X(w)$ consists of all measurable functions 
$f:\mathbb{T}\to\mathbb{C}$ such that $fw\in X$. We equip it with the norm 
\[
\|f\|_{X(w)}=\|fw\|_X. 
\]
We will suppose that $w\in X$ and $1/w\in X'$. 
Then $X(w)$ is a Banach function space itself and
$L^\infty\hookrightarrow X(w)\hookrightarrow L^1$ 
(see \cite[Lemma~2.3(b)]{K03}).

For $f\in X$ we will use the following notation:
\[
(\tau_\vartheta f)(e^{it}):=f(e^{i(t-\vartheta)}),
\quad
t,\vartheta\in[-\pi,\pi].
\]
A Banach function space is said to be translation-invariant if for every
$f\in X$ and every $\vartheta\in[-\pi,\pi]$, one has $\tau_\vartheta f\in X$
and $\|\tau_\vartheta f\|_X=\|f\|_X$. Note that all rearrangement-invariant
Banach function spaces (see \cite[Ch.~2]{BS88}) are translation-invariant.

The following analogue of \eqref{eq:Shargorodsky-estimates} holds for the 
spaces $H[X(w)]$.
%%%----------------------------------------------------------------------------
\begin{theorem}
\label{th:BCAP-DCAP-weighted-abstract-Hardy}
Let $X$ be a translation-invariant Banach function space with the associate
space $X'$ and let $w$ be a weight such that $w\in X$ and $1/w\in X'$.
%%%
\begin{enumerate}
\item[{\rm(a)}]
If $X$ is separable, then the abstract Hardy space has the BCAP with 
\[
M(H[X(w)]) \le 2.
\] 

\item[{\rm(b)}]
If $X$ is reflexive, then the abstract Hardy space $H[X(w)]$ has
the BCAP and the DCAP with
\[
M(H[X(w)])\le 2,
\quad
M^*(H[X(w)]) \le 2.
\]
\end{enumerate}
\end{theorem}
%%%----------------------------------------------------------------------------
The paper is organised as follows. In Section~\ref{sec:preliminaries}, we 
collect preliminaries on Banach function spaces. Further, we give some 
estimates for the adjoints to restrictions of operators.

In Section~\ref{sec:BCAP-DCAP-unweighted}, we show that if a 
translation-invariant Banach function space $X$ is separable, then the Hardy 
space $H[X]$ has the BCAP with $M(H[X])\le 2$. Moreover, if $X$ is reflexive, 
then $H[X]$ has the BCAP and the DCAP with $M(H[X])\le 2$ and $M^*(H[X])\le 2$,
respectively.

In Section~\ref{sec:BCAP-DCAP-weighted}, we observe that if $X$ is a Banach
function space and $w$ is a weight such that $w\in X$ and $1/w\in X'$, then
the Hardy spaces $H[X]$ and $H[X(w)]$ are isometrically isomorphic. This result 
combined with \eqref{eq:Shargorodsky-estimates} and the main result of 
Section~\ref{sec:BCAP-DCAP-unweighted} implies 
Theorems~\ref{th:BCAP-DCAP-weighted-Hardy} 
and~\ref{th:BCAP-DCAP-weighted-abstract-Hardy}.

The main part of the paper is Section~\ref{sec:question}, where some open problems
concerning approximation properties of 
Hardy spaces are stated and discussed.
%%%----------------------------------------------------------------------------
\section{Preliminaries}\label{sec:preliminaries}
\subsection{Banach function spaces}\label{sec:BFS}
Let $\mathcal{M}$ be the set of all measurable extended
complex-valued functions on 
$\mathbb{T}$ equipped with the normalized measure $dm(t)=|dt|/(2\pi)$ and let 
$\mathcal{M}^+$ be the subset of functions in $\mathcal{M}$ whose values lie 
in $[0,\infty]$. 

Following \cite[Ch.~1, Definition~1.1]{BS88}, a mapping 
$\rho: \mathcal{M}^+\to [0,\infty]$ is called a Banach function norm if, 
for all functions $f,g, f_n\in \mathcal{M}^+$ with $n\in\mathbb{N}$, and for 
all constants $a\ge 0$, the following  properties hold:
%%%
\begin{eqnarray*}
{\rm (A1)} & &
\rho(f)=0  \Leftrightarrow  f=0\ \mbox{a.e.},
\
\rho(af)=a\rho(f),
\
\rho(f+g) \le \rho(f)+\rho(g),\\
{\rm (A2)} & &0\le g \le f \ \mbox{a.e.} \ \Rightarrow \ 
\rho(g) \le \rho(f)
\quad\mbox{(the lattice property)},\\
{\rm (A3)} & &0\le f_n \uparrow f \ \mbox{a.e.} \ \Rightarrow \
       \rho(f_n) \uparrow \rho(f)\quad\mbox{(the Fatou property)},\\
{\rm (A4)} & & \rho(\mathbbm{1}) <\infty,\\
{\rm (A5)} & &\int_{\mathbb{T}} f(t)\,dm(t) \le C\rho(f)
\end{eqnarray*}
%%%%
with {a constant} $C \in (0,\infty)$ that may depend on $\rho$,  but is 
independent of $f$. When functions differing only on a set of measure zero 
are identified, the set $X$ of all functions $f\in \mathcal{M}$ for which 
$\rho(|f|)<\infty$ is called a Banach function space. For each $f\in X$, 
the norm of $f$ is defined by $\|f\|_X :=\rho(|f|)$. The set $X$ equipped 
with the natural linear space operations and with this norm becomes a Banach 
space (see \cite[Ch.~1, Theorems~1.4 and~1.6]{BS88}). If $\rho$ is a Banach 
function norm, its associate norm $\rho'$ is defined on $\mathcal{M}^+$ by
\[
\rho'(g):=\sup\left\{
\int_\mathbb{T} f(t)g(t)\,dm(t) \ : \ 
f\in\mathcal{M}^+, \ \rho(f) \le 1
\right\}, \quad g\in \mathcal{M}^+.
\]
It is a Banach function norm itself (\cite[Ch.~1, Theorem~2.2]{BS88}).
The Banach function space $X'$ defined by the Banach function 
norm $\rho'$ is called the associate space (K\"othe dual) of $X$. 
The associate space $X'$ can be viewed as a subspace of the 
Banach dual space $X^*$ (see \cite[Ch.~1, Theorem~2.9]{BS88}). 
The following lemma can be proved as in the non-periodic case
(see \cite[Lemma~2.1]{KS19-PLMS}).
%%%----------------------------------------------------------------------------
\begin{lemma}\label{le:TI-X-and-Xprime}
Let $X$ be a Banach function space and $X'$ be its associate space. Then $X$ 
is translation-invariant if and only if $X'$ is translation-invariant.
\end{lemma}
%%%----------------------------------------------------------------------------
\subsection{Adjoints to restrictions of operators}
In this subsection, we present some simple results, for which we could not find a 
convenient reference.

Let $X$ and $Y$ be Banach spaces, $X_0 \subseteq X$ and $Y_0 \subseteq Y$ be 
closed linear subspaces, and let $A \in \mathcal{B}(X, Y)$ be such that 
$A(X_0) \subseteq Y_0$. Let $A_0 \in \mathcal{B}(X_0, Y_0)$ be the restriction 
of $A$ to $X_0$:
\[
A_0 x_0 := Ax_0 \in Y_0 \ \mbox{ for all } \ x_0 \in X_0 .
\]
Let
\[
X_0^\perp := \{x^* \in X^* : \ x^*(x_0) = 0 \ \mbox{ for all } \ x_0 \in X_0\}
\]
and let $Y_0^\perp$ be defined similarly. Then $X_0^*$ and $Y_0^*$ are 
isometrically isomorphic to the quotient spaces $X^*/X_0^\perp$ and 
$Y^*/Y_0^\perp$, respectively (see, e.g., \cite[Theorem~7.1]{D70}). We will 
identify these spaces and will denote by $[x^*]$ the element of 
$X^*/X_0^\perp$ corresponding to $x^* \in X^*$, and similarly for $[y^*]$.
  
It is easy to see that $A^*(Y_0^\perp) \subseteq X_0^\perp$. Indeed, take any 
$y^*_0 \in Y_0^\perp$ and $x_0 \in X_0$. Since $Ax_0 \in Y_0$, one has
\[
(A^*y^*_0)(x_0) = y^*_0(Ax_0) = 0 .
\]
So, $A^*y^*_0 \in X_0^\perp$. Hence the operator $[A^*]$,
\[
[A^*] [y^*] := [A^* y^*] \in X^*/X_0^\perp , \quad [y^*] \in Y^*/Y_0^\perp
\]
is a well defined element of 
$\mathcal{B}(Y^*/Y_0^\perp, X^*/X_0^\perp) = \mathcal{B}(Y^*_0, X^*_0)$, and
it is easy to see that $A_0^* = [A^*]$. Indeed, one has for every 
$[y^*] \in Y^*/Y_0^\perp$ and $x_0 \in X_0$,
%%%
\begin{align*}
(A_0^*[y^*])(x_0) & = [y^*](A_0x_0) = [y^*](Ax_0) = y^*(Ax_0) = (A^*y^*)(x_0) 
\\
& = [A^*y^*](x_0) = ([A^*] [y^*])(x_0).
\end{align*}
%%%
%%%----------------------------------------------------------------------------
\begin{lemma}\label{le:adj-restr}
Let $X$ and $Y$ be Banach spaces, $X_0\subseteq X$ and $Y_0\subseteq Y$ be
closed linear subspaces, and $A\in\mathcal{B}(X,Y)$ be such that
$A(X_0)\subseteq Y_0$. If $A_0:=A|_{X_0}$, then for every $y^*\in Y^*$,
one has
\[
\|A_0^*[y^*]\|_{X_0^*}
\le 
\|A^*y^*\|_{X^*},
\]
where $[y^*]$ is the element of $Y^*/Y_0^\perp$ corresponding to $y^*$.
\end{lemma}
%%%----------------------------------------------------------------------------
\begin{proof}
We have
%%%
\begin{align*}
\|A_0^*[y^*]\|_{X_0^*} 
& = 
\|A_0^*[y^*]\|_{X^*/X_0^\perp} =  \|[A^*][y^*]\|_{X^*/X_0^\perp} 
=  
\|[A^*y^*]\|_{X^*/X_0^\perp} 
\nonumber \\
& = \inf_{x_0 \in X_0} \|A^*y^* + x_0\|_{X^*} \le \|A^*y^*\|_{X^*}\, .
\end{align*}
%%%
which completes the proof.
\end{proof}
%%%----------------------------------------------------------------------------
\section{Bounded compact and dual compact approximation properties 
of abstract Hardy spaces built upon translation-invariant spaces}
\label{sec:BCAP-DCAP-unweighted}
%%%----------------------------------------------------------------------------
\subsection{Continuity of shifts in separable translation-invariant 
Banach function spaces}
We start with the following simple lemma.
%%%----------------------------------------------------------------------------
\begin{lemma}\label{le:continuity-translations}
Let $X$ be a translation-invariant Banach function space. If $X$ is separable,
then for every $f\in X$,
\begin{equation}\label{eq:continuity-translations}
\lim_{\vartheta\to 0}\|\tau_\vartheta f-f\|_X=0.
\end{equation}
\end{lemma}
%%%----------------------------------------------------------------------------
\begin{proof}
By \cite[Lemma~2.2.1]{K17}, a Banach function space $X$ is separable if and 
only if the set of continuous functions $C$ is dense in $X$. Let $f\in X$
and $\varepsilon>0$. Then there exists $g\in C$ such that 
$\|f-g\|_X<\varepsilon/3$. Taking into account that $X$ is 
translation-invariant, we see that for all $\vartheta\in[-\pi,\pi]$,
%%%
\begin{align*}
\|\tau_\vartheta f-f\|_X
&\le 
\|\tau_\vartheta f-\tau_\vartheta g\|_X
+
\|\tau_\vartheta g-g\|_X
+
\|g-f\|_X
\\
&=
2\|f-g\|_X+\|\tau_\vartheta g-g\|_X
\\
&<\frac{2}{3}\varepsilon+\|\mathbbm{1}\|_X\|\tau_\vartheta g-g\|_C.
\end{align*}
%%%
Since
\[
\lim_{\vartheta\to 0}\|\tau_\vartheta g-g\|_C=0,
\]
the above inequality yields
\[
\limsup_{\vartheta\to 0}\|\tau_\vartheta f-f\|_X
\le\frac{2}{3}\varepsilon<\varepsilon.
\]
Letting $\varepsilon\to 0$, we arrive at 
\eqref{eq:continuity-translations}.
\end{proof}
%%%----------------------------------------------------------------------------
\subsection{Convolutions with integrable functions on translation-invariant 
Banach function spaces}
Recall that the convolution of two functions $f,g\in L^1$
is defined by
\[
(f*g)(e^{i\varphi}):=\frac{1}{2\pi}
\int_{-\pi}^\pi f(e^{i(\varphi-\theta)})g(e^{i\theta})\,d\theta.
\]

The following lemmas might be known to experts, however we were not able to 
find an explicit reference.
%%%----------------------------------------------------------------------------
\begin{lemma}\label{le:convolution}
Suppose that $X$ is a translation-invariant Banach function spaces. 
If $K \in L^1$, then the convolution operator $C_K$ defined by
%%%
\begin{equation}\label{eq:convolution-0}
C_K g =K*g,
\quad g\in X,
\end{equation}
%%%
is bounded on $X$ and 
%%%
\begin{equation}\label{eq:convolution-1}
\|C_K\|_{\mathcal{B}(X)}\le \|K\|_{L^1}.
\end{equation}
%%%
If, in addition, $K\ge 0$, then 
%%%
\begin{equation}\label{eq:convolution-2}
\|C_K\|_{\mathcal{B}(X)}=\|K\|_{L^1}.
\end{equation}
%%%
\end{lemma}
%%%----------------------------------------------------------------------------
\begin{proof}
For every $h \in X'$, in view of Tonelli's theorem
(see, e.g., \cite[Theorem~5.28]{A20}) and H\"older's inequality for Banach
function spaces (see \cite[Ch.~1, Theorem~2.4]{BS88}), one has
%%%
\begin{align}
& 
\int_{-\pi}^\pi\left|
(K\ast g)(e^{i \vartheta}) 
h(e^{i \vartheta})\right|\, d\vartheta
\le 
\frac1{2\pi} \int_{-\pi}^\pi\int_{-\pi}^\pi 
\left|K(e^{i (\vartheta - \theta)})\right| 
\left|g(e^{i \theta})\right| 
\left|h(e^{i \vartheta})\right|\, d\theta\, d\vartheta 
\nonumber\\
&\qquad= 
\frac1{2\pi} \int_{-\pi}^\pi\int_{-\pi}^\pi 
\left|K(e^{i \theta})\right| 
\left|g(e^{i (\vartheta - \theta)})\right| 
\left|h(e^{i \vartheta})\right|\, d\theta\, d\vartheta 
\nonumber\\
&\qquad= 
\frac1{2\pi} \int_{-\pi}^\pi \left|K(e^{i \theta})\right| 
\left(\int_{-\pi}^\pi 
\left|(\tau_\theta g)(e^{i \vartheta})\right| 
\left|h(e^{i \vartheta})\right|\, d\vartheta\right) d\theta 
\nonumber\\
&\qquad \le 
 \int_{-\pi}^\pi \left|K(e^{i \theta})\right| 
\|\tau_\theta g\|_{X} \|h\|_{X'}\, d\theta 
= 
2\pi\|K\|_{L^1} \|g\|_{X}  \|h\|_{X'} .
\label{eq:convolution-3}
\end{align}
%%%
In view of the Lorentz-Luxemburg theorem (see 
\cite[Ch.~1, Theorem~2.7]{BS88}), the last inequality implies that
%%%
\begin{align*}
\|K*g\|_{X}
&=\|K*g\|_{X''}
\\
&=
\sup\left\{
\frac{1}{2\pi}\int_{-\pi}^\pi\left|
(K\ast g)(e^{i \vartheta}) 
h(e^{i \vartheta})\right|\, d\vartheta: \ h\in X',
\|h\|_{X'}\le 1
\right\}
\\
&\le\|K\|_{L^1} \|g\|_{X},
\end{align*}
%%%
which implies \eqref{eq:convolution-1}.

If, in addition, we suppose that $K\ge 0$, then
for a.e. $\varphi\in [-\pi,\pi]$,
\[
(K*\mathbbm{1})(e^{i\varphi})
=
\frac{1}{2\pi}\int_{-\pi}^{\pi}K(e^{i(\varphi-\theta)})\,d\theta
=
\frac{1}{2\pi}\int_{-\pi}^{\pi}K(e^{i t})\,dt
=
\|K\|_{L^1}.
\]
Hence,
\[
\|C_K\|_{\mathcal{B}(X)}
=
\sup_{f\in X\setminus\{0\}}\frac{\|K*f\|_X}{\|f\|_X}
\ge 
\frac{\|K*\mathbbm{1}\|_X}{\|\mathbbm{1}\|_X}
=
\frac{\|K\|_{L^1}\|\mathbbm{1}\|_X}{\|\mathbbm{1}\|_X}=\|K\|_{L^1}.
\]
Combining this inequality with \eqref{eq:convolution-1}, we arrive
at \eqref{eq:convolution-2}.
\end{proof}
%%%----------------------------------------------------------------------------
\subsection{BCAP and DCAP of abstract Hardy spaces built upon
trans\-la\-tion-in\-vari\-ant Banach function spaces}
Now we are in a position to prove the main result of this section.
%%%----------------------------------------------------------------------------
\begin{theorem}\label{th:BCAP-DCAP-abstract-Hardy}
Let $X$ be a translation-invariant Banach function space.
\begin{enumerate}
\item[{\rm(a)}]
If $X$ is separable, then the abstract Hardy space $H[X]$ has the BCAP with
\[
M(H[X])\le 2.
\]

\item[{\rm(b)}]
If $X$ is reflexive, then the abstract Hardy space $H[X]$ has the BCAP
and DCAP with
\[
M(H[X])\le 2,
\quad
M^*(H[X])\le 2.
\]
\end{enumerate}
\end{theorem}
%%%----------------------------------------------------------------------------
\begin{proof}
(a) For $\theta \in [-\pi, \pi]$ and $n = 0, 1, 2, \dots$, let
\[
K_n\left(e^{i\theta}\right) 
:= 
\sum_{k = -n}^n \left(1 - \frac{|k|}{n + 1}\right) e^{i k \theta} 
= 
\frac1{n + 1}
\left(\frac{\sin\frac{(n + 1)\theta}{2}}{\sin\frac{\theta}2}\right)^2 ,
\]
be the $n$-th Fej\'er kernel, and let
\[
\mathbf{K}_n f
:= 
K_n *f,
\quad
f\in X.
\]
It is well known that $K_n\ge 0$, $\|K_n\|_{L^1} = 1$, and
%%%
\begin{equation}\label{eq:Fejer}
\left(\mathbf{ K}_n f\right)\left(e^{i\vartheta}\right) 
= 
\sum_{k = -n}^n 
\widehat{f}(k) \left(1 - \frac{|k|}{n + 1}\right) e^{i k \theta} ,
\end{equation}
%%%
where $\widehat{f}(k)$ is the $k$-th Fourier coefficient of $f$
(see, e.g., \cite[Ch.~I, Section~2.5]{K04}). It follows from
Lemma~\ref{le:convolution} that $\|\mathbf{K}_n\|_{X \to X} = 1$.
Hence 
\[
\|I - \mathbf{ K}_n\|_{\mathcal{B}(X)} 
\le 
1 + \|\mathbf{ K}_n\|_{\mathcal{B}(X)} = 2.
\]
It follows from Lemma~\ref{le:continuity-translations} that
a separable translation-invariant Banach function space $X$
is a homogeneous Banach space in the sense of
\cite[Ch.~I, Definition~2.10]{K04}. Hence
\cite[Ch.~I, Theorem~2.11]{K04} implies that $\mathbf{ K}_n$ 
converge strongly to the identity operator on $X$ as 
$n \to \infty$. Moreover, \eqref{eq:Fejer} implies that
$\mathbf{K}_n$ maps $H[X]$ to $H[X]$. Thus $M(H[X]) \le 2$.

(b) If $X$ is reflexive, then $X^* = X'$ is also separable 
(see \cite[Ch. 1, Corollaries~4.3-4.4 and 5.6]{BS88}) and 
translation-invariant (see Lemma~\ref{le:TI-X-and-Xprime}).
It follows from the above that the adjoint operators 
$\mathbf{ K}_n^* = \mathbf{ K}_n : X' \to X'$ converge strongly to the 
identity operator as $n \to \infty$. Applying Lemma~\ref{le:adj-restr} to 
$A = I - \mathbf{ K}_n$, $X_0 = Y_0 = H[X]$, one concludes that the adjoint 
operators $\mathbf{ K}_n^* : (H[X])^* \to (H[X])^*$ also converge strongly 
to the identity operator as $n \to \infty$. Hence $M^*(H[X]) \le 2$.
\end{proof}
%%%----------------------------------------------------------------------------
\section{Proofs of Theorems \ref{th:BCAP-DCAP-weighted-Hardy} and \ref{th:BCAP-DCAP-weighted-abstract-Hardy}}
\label{sec:BCAP-DCAP-weighted}
\subsection{BCAP and DCAP of isometrically isomorphic Banach spaces}
The next lemma follows immediately form the definitions of the BCAP and the 
DCAP.
%%%----------------------------------------------------------------------------
\begin{lemma}\label{le:BCAP-DCAP-isometry}
Let $E$ and $F$ be isometrically isomorphic Banach spaces. 
\begin{enumerate}
\item[{\rm(a)}]
The space $E$ has the BCAP if and only if $F$ has the BCAP. In this case 
\[
M(E)=M(F).
\]

\item[{\rm(b)}]
The space $E$ has the DCAP if and only if $F$ has the DCAP. In this case 
\[
M^*(E)=M^*(F).
\]
\end{enumerate}
\end{lemma}
%%%----------------------------------------------------------------------------
\subsection{Isometric isomorphism of weighted and nonweighted abstract Hardy spaces}
Having in mind the previous lemma, we show that $H[X]$ and $H[X(w)]$
are isometrically isomorphic under natural assumptions on weights $w$.
%%%----------------------------------------------------------------------------
\begin{lemma}\label{le:Hardy-isometric-isomorphism} 
Let $X$ be a Banach function space with the associate space $X'$ and let
$w$ be a weight such that $w\in X$ and $1/w\in X'$. Then $H[X(w)]$ is 
isometrically isomorphic to $H[X]$. 
\end{lemma}
%%%----------------------------------------------------------------------------
\begin{proof}
Let $\mathbb{D}$ be the unit disc: 
$\mathbb{D} :=\left\{z \in \mathbb{C} : \ |z| < 1\right\}$. A function $F$ 
analytic in $\mathbb{D}$ is said to belong to the Hardy
space $H^p(\mathbb{D})$, $0<p\le\infty$, if the integral mean
%%%
\begin{align*}
&
M_p(r,F)
=
\left(\frac{1}{2\pi}\int_{-\pi}^\pi |F(re^{i\theta})|^p\,d\theta\right)^{1/p},
\quad
0<p<\infty,
\\
&
M_\infty(r,F)=\max_{-\pi\le\theta\le\pi}|F(re^{i\theta})|,
\end{align*}
%%%
remains bounded as $r\to 1$. If $F\in H^p(\mathbb{D})$, $0<p\le\infty$, then
the nontangential limit $F(e^{i\theta})$ exists almost everywhere on 
$\mathbb{T}$ and $F \in L^p(\mathbb{T})$ (see, e.g., \cite[Theorem 2.2]{D70}). 
If $1 \le p \le \infty$, then $F \in H^p$ (see, e.g., 
\cite[Theorem~3.4]{D70}). 

It follows from $w\in X$, $1/w\in X'$ and Axiom (A5) that $w \in L^1$, 
$\frac{1}{w} \in L^1$. Then $\log w  \in L^1$.
Consider the outer function
%%%
\[
W(z) := \exp\left(\frac1{2\pi}\int_{-\pi}^\pi
\frac{e^{it} + z}{e^{it} - z}\, \log w(e^{it})\, dt\right), 
\quad 
z \in \mathbb{D}
\]
%%%
(see \cite[Ch. 5]{H88}). It belongs to $H^1(\mathbb{D})$ and $|W| = w$ a.e. 
on $\mathbb{T}$.

It follows from the definition of $X(w)$ that
%%%
\begin{equation}\label{Weight}
\|Wf\|_X = \|wf\|_X = \|f\|_{X(w)} \quad\mbox{for all}\quad f \in H[X(w)] .
\end{equation}
%%%
Since $X(w)$ is a Banach function space, Axiom (A5) implies that 
$X(w) \subseteq L^1$ and $H[X(w)] \subseteq H^1$.
Take any $f \in H[X(w)] $. Let $F \in H^1(\mathbb{D})$ be its analytic 
extensions to the unit disk $\mathbb{D}$ by means of the Poisson integral 
(see the proof of \cite[Theorem~3.4]{D70}).  
Since $W, F \subseteq H^1(\mathbb{D})$,  H\"older's inequality implies that
$WF \in H^{1/2}(\mathbb{D})$. It follows from \eqref{Weight} and Axiom (A5) 
that $Wf \in X \subseteq L^1$.
Hence $WF \in H^1(\mathbb{D})$ (see \cite[Theorem~2.11]{D70}). So, 
$Wf \in H^1\cap X = H[X]$. This proves that
the mapping $f \mapsto Wf$ is an isometric isomorphism of $H[X(w)]$ into 
$H[X]$. 

Repeating the above argument, one gets that the mapping 
$g \mapsto \frac1W\, g$ is an isometric isomorphism of $H[X]$ into $H[X(w)]$. 
Hence $H[X(w)]$ and $H[X]$ are isometrically isomorphic.
\end{proof}
%%%%-----------------------------------------------------------------------------
\subsection{Proof of Theorem~\ref{th:BCAP-DCAP-weighted-Hardy}}
By Lemma~\ref{le:Hardy-isometric-isomorphism}, the spaces $H^p$ and $H^p(w)$
are isometrically isomorphic. Therefore, in view of 
\eqref{eq:Shargorodsky-estimates} and Lemma~\ref{le:BCAP-DCAP-isometry}, 
the weighted Hardy space has the BCAP and the DCAP and 
\[
M(H^p(w))=M(H^p)\le 2^{|1-2/p|},
\quad
M^*(H^p(w))=M^*(H^p)\le 2^{|1-2/p|},
\]
which completes the proof.
\qed
%%%----------------------------------------------------------------------------
\subsection{Proof of Theorem~\ref{th:BCAP-DCAP-weighted-abstract-Hardy}}
It follows from Lemma~\ref{le:Hardy-isometric-isomorphism} that the spaces 
$H[X]$ and $H[X(w)]$ are isometrically isomorphic. Now part (a) 
(resp., part (b)) follows from part (a) (resp., part (b))
of Lemma~\ref{le:BCAP-DCAP-isometry} and part (a) (resp., part (b)) 
of Theorem~\ref{th:BCAP-DCAP-abstract-Hardy}. 
\qed
%%%----------------------------------------------------------------------------
\section{Concluding remarks and open problems}\label{sec:question}
\subsection{Exact values of the norms of the operators
$I-\mathbf{K}_n$ and $I-\mathbf{P}_r$ on $L^p$ and $H^p$}
Upper estimates for the norms of the operators $I-\mathbf{K}_n$ play a crucial
role in the proof of estimates \eqref{eq:Shargorodsky-estimates}
(see \cite{S21}). Consider also the operators $I-\mathbf{P}_r$, where
\[
\mathbf{P}_r f := P_r*f, \quad 0\le r<1,
\]
and $P_r$ is the Poisson kernel
\[
P_r(e^{i\theta}) 
:= 
\sum_{k = -\infty}^\infty r^{|k|} e^{ik\theta} 
= 
\frac{1 - r^2}{1 + r^2 - 2r\cos\theta}, 
\quad 
\theta \in [-\pi, \pi], 
\quad 
0 \le r < 1.
\]
The following theorem provides a two-sided estimate for 
operators of this type.
%%%----------------------------------------------------------------------------
\begin{theorem}\label{th:two-sided}
Let $K \in L^1$, $\|K\|_{L^1} = 1$, $K\ge 0$, and $\widehat{K}(n) \ge 0$ for 
all $n \in \mathbb{Z}$. Then the following estimate holds for the convolution 
operator $C_K$ defined by \eqref{eq:convolution-0}
%%%
\begin{equation}\label{eq:two-sided}
C_p \le \|I - C_K\|_{\mathcal{B}(L^p)}\le 2^{|1-2/p|}, \quad 1 \le p \le \infty,
\end{equation}
%%%
where $C_1=2 =C_\infty$ and $C_p$ is given by \eqref{eq:Franchetti} for 
$p \in (1, \infty)$.
\end{theorem}
%%%----------------------------------------------------------------------------
\begin{proof}
It follows from Lemma~\ref{le:convolution} that 
$\|C_K\|_{\mathcal{B}(L^p)} = 1$, and  hence
\[
\|I-C_K\|_{\mathcal{B}(L^1)}\le 2,
\quad
\|I-C_K\|_{\mathcal{B}(L^\infty)} \le 2
\]
(cf. the proof of Theorem~\ref{th:BCAP-DCAP-abstract-Hardy}). Since
$\widehat{K}(n) \ge 0$ and $\widehat{K}(n) \le \|K\|_{L^1} = 1$, 
$n \in \mathbb{Z}$, the Parseval theorem gives 
$\|I-C_K\|_{\mathcal{B}(L^2)} \le 1$. (In fact, one can easily see that 
$\|I-C_K\|_{\mathcal{B}(L^2)} = 1$, since $\widehat{K}(n) \to 0$ as 
$n \to \infty$ due the to Riemann-Lebesgue lemma.)
Then the Riesz-Thorin interpolation theorem implies that
%%%
\begin{equation}\label{eq:identity-minus-Fejer}
\|I-C_K\|_{\mathcal{B}(L^p)}\le 2^{|1-2/p|},
\quad
1<p<\infty,
\end{equation}
%%%
which proves the upper estimate in \eqref{eq:two-sided}.

Since trigonometric polynomials are dense in $L^1$, it follows from 
Lemma \ref{le:convolution} that $C_K$ can be approximated in norm by 
finite rank operators. So, $C_K : L^p \to L^p$ is a compact operator. 
The equality
\[
\frac{1}{2\pi}
\int_{-\pi}^\pi K(e^{i(\varphi-\theta)})\cdot \mathbbm{1}\,d\theta 
= 
\|K\|_{L^1} = 1
\]
implies that $C_K$ preserves constant functions. Then 
\[
\|I-C_K\|_{\mathcal{B}(L^p)}\ge C_p,
\quad
1\le p<\infty,
\]
(see \cite[Theorem~3.4]{SS23}). 

It is left to prove the lower estimate in \eqref{eq:two-sided} for 
$p = \infty$. In the case $p = \infty$ or $p = 1$, \eqref{eq:two-sided} turns 
into the equality $\|I-C_K\|_{\mathcal{B}(L^p)} = 2$, which 
follows from Lemma~\ref{le:convolution} and
the fact that $L^\infty$ and $L^1$ have the Daugavet property 
(see \cite[Theorem~1 and the references therein]{A91} and 
\cite[Corollary~6 and its proof]{W96}):
$\|I - T\|_{\mathcal{B}(L^p)} = 1 + \|T\|_{\mathcal{B}(L^p)}$ 
for every operator $T \in \mathcal{K}(L^p)$, $p = \infty$ or $p=1$.
\end{proof}
%%%----------------------------------------------------------------------------
It is easy to see that $\mathbf{K}_n$ and $\mathbf{P}_r$ satisfy the 
conditions of Theorem~\ref{th:two-sided} and map $H^p$ into itself. Clearly,
\[
\|I-\mathbf{K}_n\|_{\mathcal{B}(H^p)}
\le
\|I-\mathbf{K}_n\|_{\mathcal{B}(L^p)}, 
\ 
\|I-\mathbf{P}_r\|_{\mathcal{B}(H^p)}
\le
\|I-\mathbf{P}_r\|_{\mathcal{B}(L^p)},
\
1\le p\le\infty.
\]
The above remarks lead to the following.
%%%----------------------------------------------------------------------------
\begin{problem}\label{Probl1}
Let $n\in\mathbb{Z}_+$ and $r \in [0, 1)$. Find the exact values of 
$\|I-\mathbf{K}_n\|_{\mathcal{B}(L^p)}$ and 
$\|I-\mathbf{P}_r\|_{\mathcal{B}(L^p)}$ for $1< p < \infty$, 
and of $\|I-\mathbf{K}_n\|_{\mathcal{B}(H^p)}$ and 
$\|I-\mathbf{P}_r\|_{\mathcal{B}(H^p)}$ for $1\le p \le \infty$.
\end{problem}
%%%----------------------------------------------------------------------------
It seems that the above problem is open even for $n=1$. For $n = 0$, one has 
$(I-\mathbf{K}_0)f = f - \widehat{f}(0) = (I-\mathbf{P}_0)f$ and
\[
\|I-\mathbf{K}_0\|_{\mathcal{B}(L^p)} = C_p
\]
(see \cite[formula (8)]{F90}), 
but the value of $\|I-\mathbf{K}_0\|_{\mathcal{B}(H^p)}$
does not seem to be known for $p \in [1, \infty)\setminus\{2\}$. What is known 
is that 
\begin{equation}\label{eq:norm-of-I-K0}
\|I-\mathbf{K}_0\|_{\mathcal{B}(H^\infty)} = 2 
\end{equation}
%%%
(see \cite[Theorem~2.5]{F17}) and
\[
\|I-\mathbf{K}_0\|_{\mathcal{B}(H^p)} < \|I-\mathbf{K}_0\|_{\mathcal{B}(L^p)} 
\]
for sufficiently small $p \ge 1$. Indeed, 
$\|I-\mathbf{K}_0\|_{\mathcal{B}(L^p)} = C_p \to 2$ as $p \to 1$, while 
\[
\|I-\mathbf{K}_0\|_{\mathcal{B}(H^p)} < 1.7047 
\]
for sufficiently small $p \ge 1$ (see the proof of \cite[Theorem~2.4]{F17}).

It follows from the lower estimate in \eqref{eq:two-sided} that
%%%
\begin{equation}\label{eq:lowern}
\|I-\mathbf{K}_n\|_{\mathcal{B}(L^p)} \ge \|I-\mathbf{K}_0\|_{\mathcal{B}(L^p)} .
\end{equation}
%%%
An analogue of this estimate holds in the $H^p$ setting.
%%%----------------------------------------------------------------------------
\begin{lemma}\label{le:n-dominates-0}
For every $n\in\mathbb{Z}_+$,
\[
\|I-\mathbf{K}_n\|_{\mathcal{B}(H^p)} \ge \|I-\mathbf{K}_0\|_{\mathcal{B}(H^p)}.
\]
\end{lemma}
%%%----------------------------------------------------------------------------
\begin{proof}
Take any $f \in H^p\setminus\{0\}$ and set 
$f_m(e^{i\theta}) := f(e^{im\theta})$, $m \in \mathbb{N}$. Then 
$f \in H^p$ and $\|f_m\|_{H^p} = \|f\|_{H^p}$ (see \cite[Theorem~5.5]{DG02}). 
Let $m>n$.
It follows from \eqref{eq:Fejer} that $\mathbf{K}_n f_m = \widehat{f}(0) = 
\mathbf{K}_0 f_m = \mathbf{K}_0 f$. Hence
%%% 
\begin{align*}
\|I-\mathbf{K}_n\|_{\mathcal{B}(H^p)} 
&= 
\sup_{g \in H^p\setminus\{0\}} 
\frac{\|(I-\mathbf{K}_n)g\|_{H^p}}{\|g\|_{H^p}}
\ge 
\sup_{f \in H^p\setminus\{0\}} 
\frac{\|(I-\mathbf{K}_n)f_m\|_{H^p}}{\|f_m\|_{H^p}} 
\\
& = 
\sup_{f \in H^p\setminus\{0\}} 
\frac{\|(I-\mathbf{K}_0)f_m\|_{H^p}}{\|f\|_{H^p}} 
= 
\sup_{f \in H^p\setminus\{0\}} 
\frac{\|(I-\mathbf{K}_0)f\|_{H^p}}{\|f\|_{H^p}} 
\\
&= 
\|I-\mathbf{K}_0\|_{\mathcal{B}(H^p)},
\end{align*}
%%%
which completes the proof.
\end{proof}
%%%----------------------------------------------------------------------------
The same argument as in the proof of Lemma~\ref{le:n-dominates-0}
applies in the $L^p$ setting and provides a simpler proof of \eqref{eq:lowern}.
%%%----------------------------------------------------------------------------
\subsection{Exact value of the norm of the backward shift operator on $H^p$}
We think that the question about the exact value of 
$\|I-\mathbf{K}_0\|_{\mathcal{B}(H^p)}$ is particularly interesting, and 
although it is a special case of Problem~\ref{Probl1}, we state it again 
below in terms of the backward shift operator 
\[
(\mathbf{B}f)(e^{i\theta}) 
:= 
e^{-i\theta} \left(f(e^{i\theta}) - \widehat{f}(0)\right)
= 
e^{-i\theta} \big((I-\mathbf{K}_0)f\big)(e^{i\theta}), \quad f \in H^p.
\] 
Clearly,
\begin{align*}
|\mathbf{B}f| = |(I-\mathbf{K}_0)f|  
\quad & \implies \quad 
\|\mathbf{B} f\|_{H^p} = \|(I-\mathbf{K}_0)f\|_{H^p} 
\ \mbox{ for all } \ f \in H^p 
\\
& \implies \quad  
\|\mathbf{ B}\|_{\mathcal{B}(H^p)} = \|I-\mathbf{K}_0\|_{\mathcal{B}(H^p)} .
\end{align*}
%%%
In particular,
\[
\|B\|_{\mathcal{B}(H^\infty)}=2
\]
(see \eqref{eq:norm-of-I-K0} and \cite[Theorem~2.5]{F17}).
%%%----------------------------------------------------------------------------
\begin{problem}\label{Probl2}
Let $1\le p < \infty$. Find the exact value of the norm 
$\|\mathbf{ B}\|_{\mathcal{B}(H^p)}$ of the backward shift operator.
\end{problem}
%%%----------------------------------------------------------------------------
\subsection{Exact values for $M(H^p)$ and $M^*(H^p)$}
It seems that estimates \eqref{eq:Shargorodsky-estimates} 
and the estimate $M(H^1)\le 2$, which follows from 
Theorem~\ref{th:BCAP-DCAP-weighted-abstract-Hardy}(a),
are all what is known about the values of $M(H^p)$ and $M^*(H^p)$. 
So, it would be interesting to get nontrivial lower and better upper 
bounds for $M(H^p)$ and $M^*(H^q)$ and, moreover, to solve the following.
%%%----------------------------------------------------------------------------
\begin{problem}
{\rm (a)}
Find the exact value of $M(H^p)$, $1\le p<\infty$.

{\rm (b)}
Find the exact value of $M^*(H^p)$, $1<p<\infty$.
\end{problem}
%%%----------------------------------------------------------------------------
Given that $M(L^p) = \|I-\mathbf{K}_0\|_{\mathcal{B}(L^p)}$ 
(see \cite[Theorem~3.2]{SS23}), it would be interesting to know whether
$M(H^p) = \|I-\mathbf{K}_0\|_{\mathcal{B}(H^p)}$.
%%%----------------------------------------------------------------------------
\subsection{Estimates for $M(H[L^\varphi])$ and $M^*(H[L^\varphi])$ in the 
case of some Orlicz spaces $L^\varphi$}
Let $\varphi:[0,\infty)\to[0,\infty]$ be a convex nondecreasing left-continuous 
function that is not identically zero or infinity on
$(0,\infty)$ and satisfies $\varphi(0)=0$. For a measurable function
$f:\mathbb{T}\to\mathbb{C}$, define
\[
I_\varphi(f):=\int_{\mathbb{T}}\varphi(|f(t)|)\,dm(t).
\]
The Orlicz space $L^\varphi$ is the set of all measurable functions 
$f:\mathbb{T}\to\mathbb{C}$ such that $I_\varphi(\lambda f)<\infty$ for
some $\lambda=\lambda(f)>0$. This space is a Banach space when equipped 
with either of the following two equivalent norms: the Luxemburg norm
\[
\|f\|_\varphi:=\inf\{\lambda>0:I_\varphi(f/\lambda)\le 1\}
\]
and the Orlicz norm (in the Amemiya form)
\[
\|f\|_\varphi^0:=\inf_{k>0}\frac{1}{k}(1+I_\varphi(kf)).
\]
It is well known that
\[
\|f\|_\varphi\le \|f\|_\varphi^0 \le 2\|f\|_\varphi
\quad\mbox{for all}\quad
f\in L^\varphi.
\]

We denote by $\mathcal{P}$ the set of all quasi-concave functions
$\rho:[0,\infty)\to[0,\infty)$, that is, the functions $\rho$ such that
$\rho(x)=0$ precisely when $x=0$, the function $\rho(x)$ is increasing
and the function $\rho(x)/x$ is decreasing on $(0,\infty)$. Let
$\widetilde{\mathcal{P}}$ denote the subset of all concave functions
in $\mathcal{P}$.

It follows from \cite[Lemma~3.2]{KM01} that if $1\le p<q\le\infty$ and 
$\rho\in\widetilde{\mathcal{P}}$, then the function $\varphi$,
inverse to the function $\varphi^{-1}$ defined by
%%%
\begin{equation}\label{eq:phi-inverse}
\varphi^{-1}(0):=0,
\quad
\varphi^{-1}(x):=x^{1/p}\rho\left(x^{1/q-1/p}\right),
\quad 
x\in(0,\infty),
\end{equation}
%%%
is convex. Moreover, if $1<p<q<\infty$, then $\varphi$ and its complementary
function $\varphi^*$ defined by
\[
\varphi^*(x):=\sup_{y>0}(xy-\varphi(y)),
\]
satisfy the $\Delta_2$-condition for all $x\ge 0$, that is,
there exist $K,K^*>0$ such that $\varphi(2x)\le K\varphi(x)$
and $\varphi^*(2x)\le K^*\varphi^*(x)$ for all $x\ge 0$. Then
$L^\varphi$ is reflexive (see, e.g., \cite[Corollary~15.4.2]{RGMP16}).

For $1<p,q<\infty$, put
\[
\gamma_{p,q}:=\inf\left\{\gamma>0:
\inf_{x+y=\gamma,\ x\ge 0,\ y\ge 0}(x^p+y^q)=1\right\}.
\]
It follows from \cite[Proposition~4.3]{KM01} that $\gamma_{p,q}$
continuously increases in $p$ and $q$. Moreover, if $p\le q$, then
\[
2^{1-1/p}\le\gamma_{p,q}\le 2^{1-1/q}.
\]
For $r\in(1,\infty)$, define $r'$ by $1/r+1/r'=1$. 
%%%----------------------------------------------------------------------------
\begin{theorem}[{\cite[Theorem~5.1]{KM01}}]
\label{th:Karlovich-Maligranda}
Let $1<p<q<\infty$ and $\rho\in\widetilde{\mathcal{P}}$. Suppose that
$\varphi^{-1}$ is defined by \eqref{eq:phi-inverse}. If $T\in\mathcal{B}(L^p)$
and $T\in\mathcal{B}(L^q)$, then $T\in\mathcal{B}(L^\varphi)$ and
\[
\|T\|_{\mathcal{B}(L^\varphi)}
\le C_{p,q}\max\left\{
\|T\|_{\mathcal{B}(L^p)},\|T\|_{\mathcal{B}(L^q)}
\right\},
\]
where $L^\varphi$ is equipped with the Luxemburg norm or with the Orlicz norm,
and
%%%
\begin{equation}\label{eq:Cpq}
1\le C_{p,q}
:=
\min\left\{
(2\gamma_{p,q})^{1/p},(2\gamma_{q',p'})^{1/q'}
\right\}
\le
2^{1/(pq')+\min\{1/p,1/q'\}}.
\end{equation}
\end{theorem}
%%%----------------------------------------------------------------------------
Using this interpolation theorem, we can refine the results of
Theorem \ref{th:BCAP-DCAP-abstract-Hardy}(b) for some Orlicz spaces.
%%%----------------------------------------------------------------------------
\begin{theorem}
Let $1<p<q<\infty$ and $\rho\in\widetilde{\mathcal{P}}$. Suppose that
$\varphi^{-1}$ is defined by \eqref{eq:phi-inverse} and the corresponding
Orlicz space $L^\varphi$ is equipped with the Luxemburg norm or with
the Orlicz norm. Then the Hardy-Orlicz space $H[L^\varphi]$ has the
BCAP and the DCAP with 
\[
M(H[L^\varphi])\le \min\{2,\Lambda_{p,q}\},
\quad
M^*(H[L^\varphi])\le \min\{2,\Lambda_{p,q}\},
\]
where
\begin{equation}\label{eq:Lambda-pq}
\Lambda_{p,q}:=C_{p,q}\max\left\{2^{|1-2/p|},2^{|1-2/q|}\right\},
\end{equation}
and the constant $C_{p,q}$ is defined by \eqref{eq:Cpq}.
\end{theorem}
%%%----------------------------------------------------------------------------
\begin{proof}
It is well-known and easy to check that each Orlicz space is 
translation-invariant. As it was mentioned above, $L^\varphi$ is reflexive under the assumptions 
of the Theorem. Therefore, by 
Theorem~\ref{th:BCAP-DCAP-abstract-Hardy}(b), the Hardy-Orlicz space
$H[L^\varphi]$ has the BCAP and the DCAP with $M(H[L^\varphi])\le 2$
and $M^*(H[L^\varphi])\le 2$. It remains to show that
%%%
\begin{equation}\label{eq:Hardy-Orlicz}
M(H[L^\varphi])\le\Lambda_{p,q},
\quad 
M^*(H[L^\varphi])\le\Lambda_{p,q}.
\end{equation}
%%%
It follows from \eqref{eq:Fejer}, \eqref{eq:identity-minus-Fejer}
and Theorem~\ref{th:Karlovich-Maligranda} that for all $n\in\mathbb{Z}_+$,
%%%
\begin{align*}
\|I-\mathbf{K}_n\|_{\mathcal{B}(H[L^\varphi])}
&\le 
\|I-\mathbf{K}_n\|_{\mathcal{B}(L^\varphi)}
\\
&\le 
C_{p,q}\max\left\{
\|I-\mathbf{K}_n\|_{\mathcal{B}(L^p)},
\|I-\mathbf{K}_n\|_{\mathcal{B}(L^q)}
\right\}
\\
&\le 
C_{p,q}\max\left\{
2^{|1-2/p|},2^{|1-2/q|}
\right\}
=
\Lambda_{p,q},
\end{align*}
%%%
where the Orlicz space $L^\varphi$ is equipped with the Luxemburg norm
or the Orlicz norm. As in the proof of 
Theorem~\ref{th:BCAP-DCAP-abstract-Hardy}(b), this implies 
\eqref{eq:Hardy-Orlicz}.
\end{proof}
%%%----------------------------------------------------------------------------
It follows from \eqref{eq:Cpq} and \eqref{eq:Lambda-pq} that if
$p$ and $q$ are sufficiently close to $2$, then $M(H[L^\varphi])<2$
and $M^*(H[L^\varphi])<2$. Given that the value of $M(H^p)$ is not known, 
it would perhaps be too ambitious to ask about the exact values of 
$M(H[L^\varphi])$ and $M^*(H[L^\varphi])$. Nevertheless, we think it would 
be interesting to get more information on these quantities. 
%%%----------------------------------------------------------------------------
\subsection{Estimates for $M(H[L^{p,q}])$ and $M^*(H[L^{p,q}])$ in the 
case of Lorentz spaces $L^{p,q}$}
The distribution function $m_f$ of a measurable a.e. finite function
$f:\mathbb{T}\to\mathbb{C}$ is given by
\[
m_f(\lambda) :=  m\{t\in\mathbb{T}:|f(t)|>\lambda\},\quad\lambda\ge 0.
\]
The non-increasing rearrangement of $f$ is
defined by
\[
f^*(x):=\inf\{\lambda:m_f(\lambda)\le x\},\quad x\ge 0.
\]
We refer to \cite[Ch.~2, Section~1]{BS88} for properties of distribution
functions and non-increasing rearrangements.

One of the closest classes of translation-invariant spaces to the class of 
Lebesgue spaces $L^p$, $1\le p\le\infty$ consists of the Lorentz 
spaces $L^{p,q}$ defined as follows.
For $1\le q\le p<\infty$, the Lorentz space $L^{p,q}$ consists of all
measurable functions $f:\mathbb{T}\to\mathbb{C}$ for which
\[
\|f\|_{p,q}:=\left(\int_0^1 [t^{1/p}f^*(t)]^q\frac{dt}{t}\right)^{1/q}<\infty.
\]
This is a rearrangement-invariant Banach function space with respect to
the norm $\|\cdot\|_{p,q}$ (see, e.g., \cite[Ch.~4, Theorem~4.3]{BS88}).
The Lorentz space $L^{p,p}$ is isometrically isomorphic to the Lebesgue 
space $L^p$.

It follows from Theorem~\ref{th:BCAP-DCAP-abstract-Hardy} that
if $1\le q\le p<\infty$, then the Hardy-Lorentz space $H[L^{p,q}]$ has the 
BCAP with 
\begin{equation}\label{eq:M-Hardy-Lorentz}
M(H[L^{p,q}])\le 2, 
\end{equation}
%%%
because $L^{p,q}$ is separable in this case. 
Moreover, if $1< q\le p<\infty$, then the Hardy-Lorentz space $H[L^{p,q}]$
has the DCAP with 
%%%
\begin{equation}\label{eq:M*-Hardy-Lorentz}
M^*(H[L^{p,q}])\le 2, 
\end{equation}
%%%
since the Lorentz space $L^{p,q}$ is reflexive in this case.

Having in mind estimates \eqref{eq:Shargorodsky-estimates}, which can
be stated as follows:
\[
M(H[L^{p,p}])\le 2^{|1-2/p|},
\quad
M^*(H[L^{p,p}])\le 2^{|1-2/p|},
\quad
1<p<\infty,
\] 
it seems natural to formulate the following.
%%%----------------------------------------------------------------------------
\begin{problem}
{\rm (a)}
Let $1\le q\le p<\infty$. Find a nontrivial lower bound for $M(H[L^{p,q}])$.
Improve the upper bound for $M(H[L^{p,q}])$ given by 
\eqref{eq:M-Hardy-Lorentz}.

{\rm(b)}
Let $1< q\le p<\infty$. Find a nontrivial lower bound for $M^*(H[L^{p,q}])$.
Improve the upper bound for $M^*(H[L^{p,q}])$ given by 
\eqref{eq:M*-Hardy-Lorentz}.
\end{problem}
%%%----------------------------------------------------------------------------
\subsection*{Acknowledgments}
This work is funded by national funds through the FCT - Funda\c{c}\~ao para a 
Ci\~encia e a Tecnologia, I.P., under the scope of the projects UIDB/00297/2020 
and UIDP/00297/2020 (Center for Mathematics and Applications).



%%\bibliographystyle{elsarticle-num} 
\bibliographystyle{abbrv}
\bibliography{OKES17}


%\begin{thebibliography}{00}


%\bibitem{}

%\end{thebibliography}
\end{document}